\documentclass[11pt]{article}

\usepackage{amsmath}
\usepackage{latexsym}
\usepackage{amsfonts}
\usepackage{amssymb}
\usepackage{amscd}
\usepackage{theorem}
\newtheorem{theorem}{Theorem}[section]

\newtheorem{lemma}[theorem]{Lemma}
\newtheorem{proposition}[theorem]{Proposition}

{\theorembodyfont{\rmfamily}

}

\newenvironment{proof}{    % De proof o de prueba seg\'{u}n convenga
  \noindent
  \textbf{Proof.}}{
  \hfill $\Box$
  \vspace{3mm}
}

\numberwithin{equation}{section}

\newcommand{\N}{\mathbb{N}} %% Conjunto naturales:     \N
\newcommand{\C}{\mathbb{C}} %% Conjunto complejos:     \C
\begin{document}

\title{The spectrum of  Volterra operators on weighted spaces of entire functions}

\author{Jos\'{e} Bonet  }

\date{}

\maketitle

\begin{abstract}
We investigate the spectrum of the Vol\-terra operator $V_g$ with symbol an entire function $g$, when it acts on weighted Banach spaces $H_v^{\infty}(\C)$ of entire functions with sup-norms and when it acts on H\"ormander algebras $A_p$ or $A^0_p$.
\end{abstract}

%% Footnotes
\renewcommand{\thefootnote}{}

%%%%%%%%%%%%%%%%%%%%%%%%%%%%%%%%%%%%%%%%%%%%%%%%%%%%%%%%%%%%%%%%%%%%%%%%%
%%%%%%%%%% AQU\'{I} EMPIEZA EL DOCUMENTO
%%%%%%%%%%%%%%%%%%%%%%%%%%%%%%%%%%%%%%%%%%%%%%%%%%%%%%%%%%%%%%%%%%%%%%%%%
%%%%%%%%%%%%%%%%%%%%%%%%%%%%%%%%%%%%%%%%%%%%%%%%%%%%%%%%%%%%%%%%%%%%%%%

\section{Introduction, notation and preliminaries}
\label{sec1}

The aim of this note is to investigate the spectrum of the Volterra operator when it acts continuously on a weighted Banach space of entire functions $H^\infty_v(\C)$. Our main result is Theorem \ref{mainBanachspectrum}. The characterizations of continuous Volterra operators on $H^\infty_v(\C)$ obtained by Taskinen and the author in \cite{BoTask} play an important role. The present research originated with a question of A.\ Aleman in a meeting celebrated in Granada in February 2015 in which the results in \cite{BoTask} were presented. In Section \ref{spectraalgebras} we study the spectrum of Volterra operators acting on a (DFN) H\"ormander algebra $A_p$ in Theorem \ref{VolterraAp}, and on a Fr\'echet H\"ormander algebra $A^0_p$ in Theorem \ref{VolterraA0p}.

In what follows $H(\C)$ denotes the space of entire functions. The space $H(\C)$ is a Fr\'echet space endowed with the compact open topology. The differentiation operator $Df(z)=f'(z)$,
the integration operator $Jf(z)=\int_0^z f(\zeta)d\zeta$ and the multiplication operator $M_h(f)= hf, h \in H(\C),$ are continuous on $H(\C).$

Given a non-constant entire function $g \in H(\C)$ with $g(0)=0$, the \textit{Volterra operator $V_g$ with symbol $g$} is defined on $H(\C)$ by $$
V_g(f)(z):= \int_0^z f(\zeta)g'(\zeta)d\zeta \ \ \ \ (z \in \C).
$$
For $g(z)=z$ this reduces to the integration operator. Clearly $V_g$ defines a continuous operator on $H(\C)$. The Volterra operator for holomorphic functions on the unit disc was introduced by Pommerenke \cite{Po} and he proved that $V_g$ is bounded on the Hardy space $H^2$, if and only if $g \in BMOA$. Aleman and Siskakis \cite{AS1} extended this result for $H^p, 1 \leq p < \infty,$ and they considered later in \cite{AS2} the case of weighted Bergman spaces; see also \cite{PauPe}. We refer the reader to the memoir by Pel\'aez and R\"atty\"a \cite{PR} and the references therein. Volterra operators on weighted Banach spaces of holomorphic functions on the disc of type $H^\infty$ have been investigated recently in \cite{BCHMP}. Constantin started in \cite{C} the study of the Volterra operator on spaces of entire functions. She characterized the continuity of $V_g$ on the classical Fock spaces and investigated its spectrum. Constantin and Pel\'aez \cite{CP} characterize the entire functions $g \in H(\C)$ such that $V_g$ is bounded or compact on a large class of Fock spaces induced by smooth radial weights. See also \cite{BoTask}. Aleman and Constantin \cite{AC} and Aleman and Pel\'aez \cite{AP} investigate the spectra of Volterra operators on several spaces of of holomorphic functions on the disc.

A {\it weight} $v$   is a continuous function  $v: [0, \infty[ \to ]0,  \infty [$, which is non-increasing on $[0,\infty[$ and satisfies
$\lim_{r \rightarrow \infty} r^m v(r)=0$ for each $m \in \N$. We extend $v$ to $\C$ by $v(z):= v(|z|)$. For
such a weight, the {\it weighted Banach space of entire functions}
is defined by
\begin{center}
$H^\infty_v(\C) := \{ f \in H(\C) \ | \  \Vert f \Vert_v := \sup_{z \in \C} v(|z|) |f(z)| <
 \infty \}$,
\end{center}
and it is  endowed with the weighted sup norm $\Vert \cdot  \Vert_v .$ Changing the value of $v$ on a compact interval does not change the space and gives an equivalent norm. Spaces of this type appear in the study of growth conditions of analytic functions and have been investigated
in various articles, see e.g. \cite{BiBoGal,BBT,Lusk} and the references
therein.

For an entire function $f \in H(\C)$, we denote $M(f,r):= \max\{|f(z)| \ | \ |z|=r\}$. Using the notation $O$ and $o$ of Landau, $f \in H_v^\infty(\C)$ if and only if $M(f,r)=O(1/v(r)), r \rightarrow \infty$. By $C$, $C'$, $c$ etc.
we denote  positive constants, the value of  which may vary from place to place.

Let $T:X \rightarrow X$ be a continuous operator on a space $X$. The \textit{resolvent} of $T$ on $X$ is the set $\rho(T,X)$ of all $\lambda \in \C$ such that $T - \lambda I: X \rightarrow X$ is bijective and has a continuous inverse. Here $I$ stands for the identity operator on $X$. The \textit{spectrum} $\sigma(T,X)$ of $T$ is the complement in $\C$ of the resolvent. The point spectrum is the set $\sigma_{pt}(T,X)$ of those $\lambda \in \C$ such that $T - \lambda I$ is not injective. Observe that in this paper we consider operators defined not only on Banach spaces, but also on more general spaces: $H(\C)$ and the H\"ormander algebra $A^0_p$ are Fr\'echet spaces and the H\"ormander algebra $A_p$ is the dual of a Fr\'echet space. Accordingly, when we refer to a space, we mean a Hausdorff locally convex space. We refer the reader to \cite{Meise_Vogt_1997_Introduction} for results and terminology about functional analysis.

At this point we include some preliminary results, in particular about the spectrum of the Volterra operator on $H(\C)$ that are inspired by \cite{AC} and \cite{AP}. Let  $g \in H(\C)$ be a non-constant entire function such that $g(0)=0$ and let $V_gf(z)= \int_0^z f(\zeta)g'(\zeta)d\zeta, \ z \in \C,$ denote the Volterra operator associated with $g$, that acts continuously on $H(\C)$.

\begin{proposition}\label{elementarypt}
The operator $V_g - \lambda I: H(\C) \rightarrow H(\C)$ is injective for each $\lambda \in \C$. In particular $\sigma_{pt}(V_g,H(\C)) = \emptyset$. Moreover, $0 \in \sigma(V_g,H(\C))$.
\end{proposition}
\begin{proof}
If $(V_g - \lambda I)f = 0$, then $\lambda f(0)=0$ and $f(z)g'(z) = \lambda f'(z), \ z \in \C,$. If $\lambda = 0$, we get $f=0$, since $g$ is not constant. In case $\lambda \neq 0$, we get $f(z) = C \exp(g(z)/\lambda)$, which implies $f=0$, since $f(0)=0$.

The operator $V_g$ is not surjective on $H(\C)$ because $V_g f(0)=0$ for each $f \in H(\C)$. Thus $0 \in \sigma(V_g,H(\C))$.
\end{proof}

\begin{lemma}\label{formresolvent}
Given $\lambda \in \C, \ \lambda \neq 0,$ and $h \in H(\C)$, the equation $f - (1/\lambda) V_g f = h$ has a unique solution given by
$$
f(z)=R_{\lambda,g} h(z) = h(0) e^{\frac{g(z)}{\lambda}} +  e^{\frac{g(z)}{\lambda}} \int_0^z e^{-\frac{g(\zeta)}{\lambda}} h'(\zeta) d\zeta, \ \ \ \ z \in \C.
$$
\end{lemma}
\begin{proof}
This is well known; see e.g. \cite[p.\ 200]{AC} or \cite[p.\ 2]{AP}. The uniqueness follows from Proposition \ref{elementarypt}. It is enough to substitute in the equation to check the result. Alternatively, the equation implies $f(0)=h(0)$ and that $f$ is the solution of the equation $y' - (g'(z)/\lambda) y = h'(z)$. The result is also obtained solving this linear equation.
\end{proof}

\begin{proposition}\label{spectrumentire}
Let $g \in H(\C)$ be a non-constant entire function such that $g(0)=0$. The Volterra operator $V_g$ satisfies $\sigma(V_g,H(\C))= \{ 0 \}$ and $\sigma_{pt}(V_g,H(\C)) = \emptyset$.
\end{proposition}
\begin{proof}
This is a direct consequence of Proposition \ref{elementarypt} and the continuity of the operator
$R_{\lambda,g}:H(\C) \rightarrow H(\C), \ \lambda \neq 0$.
\end{proof}

\begin{lemma}\label{subsetspectrumgeneral}
Let $X \subset H(\C)$ be a locally convex space that contains the constants and such that the inclusion $X \subset H(\C)$ is continuous. Assume that $V_g: X \rightarrow X$ is continuous for some non-constant entire function $g$ such that $g(0)=0$. Then
$$
\{ 0 \} \cup \{ \lambda \in \C \setminus \{ 0 \} \ | \ e^{\frac{g}{\lambda}} \notin X \} \subset \sigma(V_g,X).
$$
If $X$ is a Banach space, then
$$
\{ 0 \} \cup \overline{\{ \lambda \in \C \setminus \{ 0 \} \ | \ e^{\frac{g}{\lambda}} \notin X \}} \subset \sigma(V_g,X).
$$
\end{lemma}
\begin{proof}
Since $X$ contains the constants, $V_g$ is not surjective and $0 \in \sigma(V_g,X)$. If $\lambda \notin \sigma(V_g,X)$, then $\lambda \neq 0$ and the operator $R_{\lambda,g}: X \rightarrow X$
defined in Lemma \ref{formresolvent} is continuous. In particular, $R_{\lambda,g}(1) = e^{\frac{g}{\lambda}} \in X$. This implies the desired inclusion. Recall that in case $X$ is a Banach space, $\sigma(V_g,X)$ is compact.
\end{proof}

\begin{lemma}\label{charactresolvent}
Let $X \subset H(\C)$ be a locally convex space that contains the constants and such that the inclusion $X \rightarrow H(\C)$ is continuous. Assume that $V_g: X \rightarrow X$ is continuous for some non-constant entire function $g$ such that $g(0)=0$. The following conditions are equivalent:
\begin{itemize}

\item[(i)] $\lambda \in \rho(V_g,X)$.

\item[(ii)] $R_{\lambda,g}: X \rightarrow X$ is continuous.

\item[(iii)] (a) $e^{\frac{g}{\lambda}} \in X$, and

(b) $S_{\lambda,g}: X_0 \rightarrow X_0, \ S_{\lambda,g}h(z):= e^{\frac{g(z)}{\lambda}} \int_0^z h'(\zeta) e^{-\frac{g(\zeta)}{\lambda}} d\zeta, \  z \in \C,$ is continuous on the subspace $X_0$ of $X$ of all the functions $h \in X$ with $h(0)=0$.

\end{itemize}
\end{lemma}
\begin{proof}
This follows directly from the definitions.
\end{proof}

\begin{lemma}\label{integration by parts}
Let $X \subset H(\C)$ be a locally convex space that contains the constants and such that the inclusion $X \rightarrow H(\C)$ is continuous. Let $X_0$ be the subspace of $X$ of all the functions $h \in X$ with $h(0)=0$. The following conditions are equivalent for $\lambda \in \C\setminus \{ 0 \}$.

\begin{itemize}

\item[(i)] $S_{\lambda,g}: X_0 \rightarrow X_0, \ S_{\lambda,g}h(z):= e^{\frac{g(z)}{\lambda}} \int_0^z h'(\zeta) e^{-\frac{g(\zeta)}{\lambda}} d\zeta, \  z \in \C,$ is continuous.

\item[(ii)] $T: X_0 \rightarrow X_0, Th(z):= e^{\frac{g(z)}{\lambda}} \int_0^z h(\zeta) g'(\zeta) e^{-\frac{g(\zeta)}{\lambda}} d\zeta, \  z \in \C,$ is continuous.

\end{itemize}
\end{lemma}
\begin{proof}
This is a direct consequence of the identity
$$
e^{\frac{g(z)}{\lambda}} \int_0^z h'(\zeta) e^{-\frac{g(\zeta)}{\lambda}} d\zeta = h(z) + (1/\lambda)
e^{\frac{g(z)}{\lambda}} \int_0^z h(\zeta) g'(\zeta) e^{-\frac{g(\zeta)}{\lambda}} d\zeta,
$$
valid for $h \in H(\C), h(0)=0,$ that can be seen integrating by parts.
\end{proof}

\section{Spectra of Volterra operators on $H^\infty_v(\C)$} \label{spectraBanach}

We concentrate our attention in the Volterra operator acting on the Banach space $H^\infty_v(\C)$, with $v(r) = \exp(-\alpha r^p)$, where $\alpha, p >0$. According to \cite[Corollaries 3.12 and 3.13]{BoTask}, we have the following result.

\begin{proposition} \label{VolterraBanachcharact}
Assume that $v(r)=\exp(-\alpha r^p)$, $\alpha>0, \ p > 0$.

\begin{itemize}
\item[(i)] $V_g : H^{\infty}_v(\C) \rightarrow H^{\infty}_{v}(\C)$ is continuous if and only if $g$ is a polynomial of degree less than or equal to the integer part of $p$.

\item[(ii)] $V_g : H^{\infty}_v(\C) \rightarrow H^{\infty}_v (\C)$ is compact if and only if $g$ is a polynomial of degree less than or equal to the integer part of $p-1$.
\end{itemize}
\end{proposition}

\begin{lemma}\label{integraltechnical}
Let $v$ be a weight such that $v(r)e^{\alpha r^n}$ is non-increasing on $[r_0,\infty[$ for some $r_0>0, \ \alpha > 0$ and $n \in \N$. The operator $T: H^{\infty}_v(\C) \rightarrow H^{\infty}_{v}(\C)$ defined by
$$
T_{\gamma} h(z):= e^{\gamma z^n} \int_0^z \zeta ^{n-1} h(\zeta) e^{-\gamma \zeta^n} d\zeta, \ \ \ z \in \C,
$$
is continuous if $|\gamma| < \alpha$.
\end{lemma}
\begin{proof}
Changing the value of $v(r)$ on a compact interval, we may assume that $v(r)e^{\alpha r^n}$ is non-increasing on $[0,\infty[$. Fix $z \in \C, z \neq 0$. For each $0 \leq t \leq 1$ we have
$$
v(|z|) \leq v(t|z|) e^{\alpha|z|^n(t^n-1)}.
$$
Therefore we can estimate
$$
v(|z|) |T_{\gamma} h(z)| = v(|z|) \left| \int_0^1 z^n t^{n-1} h(tz) e^{\gamma z^n(1 - t^n)} dt \right| \leq
$$
$$
\leq |z|^n \int_0^1 t^{n-1} |h(tz)|v(tz) e^{|\gamma| |z|^n(1-t^n)} e^{\alpha |z|^n(t^n-1)} dt \leq
$$
$$
\leq ||h||_v |z|^n \int_0^1 t^{n-1} e^{(\alpha-|\gamma|) |z|^n(t^n-1)} dt =
$$
$$
(1/n) ||h||_v (\alpha-|\gamma|)^{-1} \big(1 - e^{-(\alpha-|\gamma|)|z|^n} \big) \leq (1/n) (\alpha-|\gamma|)^{-1} ||h||_v,
$$
since $\alpha-|\gamma|>0$.
\end{proof}

\begin{theorem}\label{mainBanachspectrum}
Assume that $v(r)=\exp(-\alpha r^p)$, $\alpha>0, \ p > 0$. Let $g$ be a polynomial of degree $n$ less than or equal to the integer part of $p$ with $g(0)=0$.

\begin{itemize}
\item[(i)] If the degree $n$ of $g$ satisfies $n < p$, then $\sigma(V_g,H^{\infty}_{v}(\C)) = \{ 0 \}$.

\item[(ii)] If $p=n \in \N$ and $g(z)=\beta z^n + k(z)$, $k$ a polynomial of degree strictly less than $n$, then $\sigma(V_g,H^{\infty}_{v}(\C)) = \{ \lambda \in \C \ | \ |\lambda| \leq |\beta|/\alpha \}$.
\end{itemize}

Moreover, we have $\sigma(V_g,H^{\infty}_{v}(\C)) = \{ 0 \} \cup \overline{\{ \lambda \in \C \setminus \{ 0 \} \ | \ e^{\frac{g}{\lambda}} \notin H^{\infty}_{v}(\C) \}} $.
\end{theorem}
\begin{proof}
(i) If $n$ is less than or equal to the integer part of $p-1$, then $V_g: H^{\infty}_v(\C) \rightarrow H^{\infty}_v (\C)$ is compact. This implies that $\sigma(V_g,H^{\infty}_{v}(\C)) = \{ 0 \}$, since $V_g$ has no eigenvalues by Proposition \ref{elementarypt}.

Assume now that $p-1 < n< p$. For each $\lambda \neq 0$, $e^{\frac{g}{\lambda}} \in H^{\infty}_v(\C)$ as it is easy to check.

Suppose first that $g(z)=\beta z^n$ for some $\beta \neq 0$. For $\lambda \neq 0$, take $\alpha > |\beta|/|\lambda|$. Clearly $v(r)e^{\alpha r^n}$ is non-increasing on $[r_0,\infty[$ for some $r_0>0$. We can apply Lemma \ref{integraltechnical} to conclude that
$$
T_{\beta/\lambda} h(z):= e^{\frac{\beta z^n}{\lambda}} \int_0^z \zeta ^{n-1} h(\zeta) e^{-\frac{\beta \zeta^n}{\lambda}} d\zeta, \ \ \ z \in \C,
$$
is continuous on the subspace of $H^{\infty}_{v}(\C)$ of the functions vanishing at $0$. By Lemma \ref{integration by parts}, the operator
$$S_{\lambda,g}h(z):= e^{\frac{g(z)}{\lambda}} \int_0^z h'(\zeta) e^{-\frac{g(\zeta)}{\lambda}} d\zeta, \  z \in \C,$$ is continuous in the same Banach space.
Finally, Lemma \ref{charactresolvent} implies $\lambda \in \rho(V_g,H^{\infty}_v(\C))$. This completes the proof in this case.

Suppose now that $g(z)=\beta z^n + k(z)$ for some $\beta \neq 0$ and some polynomial $k$ of degree strictly less than $n$. Setting $g_1(z):= \beta z^n$, we have $V_g=V_{g_1} + V_k$, and $V_k$ is a compact injective operator on $H^{\infty}_v(\C)$. If $\lambda \neq 0$, we have $V_g - \lambda I= (V_{g_1} - \lambda I) + V_k$, and $V_g - \lambda I$ is bijective if and only if $V_{g_1} - \lambda I$ is bijective. This is a consequence e.g.\  of \cite[Corollary 34.14]{Jameson} keeping in mind that both $V_g - \lambda I$ and $V_{g_1} - \lambda I$ are injective by Proposition \ref{elementarypt}. Therefore, we conclude $\sigma(V_g,H^{\infty}_{v}(\C)) = \sigma(V_{g_1},H^{\infty}_{v}(\C)) = \{ 0 \}$.

(ii) We suppose now that $v(r)=\exp(-\alpha r^n)$, $\alpha>0,$ and that $g$ is a polynomial of degree exactly $n$.

Again we consider first the case $g(z)=\beta z^n$. For $\lambda \in\C\setminus \{ 0 \}$, we have $e^{\frac{g}{\lambda}} \in H^{\infty}_v(\C)$ if and only if $|\beta|/|\lambda| \leq \alpha$. Therefore, we can apply Lemma \ref{subsetspectrumgeneral} to conclude that $\{ \lambda \ | \ |\lambda| \leq |\beta|/ \alpha \} \subset \sigma(V_g,H^{\infty}_{v}(\C))$. Now take $\lambda \in \C$ with $|\lambda| > |\beta|/\alpha$. Since $v(r)\exp(\alpha r^n)=1$, we apply Lemma \ref{integraltechnical} and Lemma \ref{integration by parts} to get that the operator
$$S_{\lambda,g}h(z):= e^{\frac{g(z)}{\lambda}} \int_0^z h'(\zeta) e^{-\frac{g(\zeta)}{\lambda}} d\zeta, \  z \in \C,$$ is continuous on the subspace of $H^{\infty}_{v}(\C)$ of the functions vanishing at $0$. Consequently $\lambda \in \rho(V_g,H^{\infty}_v(\C))$ by Lemma \ref{charactresolvent}, which yields
$\sigma(V_g,H^{\infty}_{v}(\C)) = \{ \lambda \in \C \ | \ |\lambda| \leq |\beta|/\alpha \}$ in the present case.

In the general case $g(z)=\beta z^n + k(z), \beta \neq 0$ and some polynomial $k$ of degree strictly less than $n$, we proceed as in the proof of part (i) to conclude $\sigma(V_g,H^{\infty}_{v}(\C)) = \sigma(V_{g_1},H^{\infty}_{v}(\C)), \ g_1(z)=\beta z^n$.

\end{proof}

Given a weight $v$, one defines the following closed subspace of $H^{\infty}_{v}(\C))$:

$$H_v^0(\C) := \{ f \in H(\C) \ | \ \lim_{|z| \rightarrow \infty} v(|z|)|f(z)|=0 \}.$$

The polynomials are contained and dense in $H_v^0(\C)$. By \cite[Ex 2.2]{bisum}, the bidual of $H^0_v(\C)$ is isometrically isomorphic to $H^\infty_v(\C).$
By \cite[Corollaries 3.8 and 3.12]{BoTask} $V_g$ is bounded on $H_v^0(\C)$ if and only if it is bounded on $H^{\infty}_{v}(\C)$ when $v(r)=\exp(-\alpha r^p), \alpha, p >0$. Now, proceeding as in the proof of \cite[Lemma 1]{BCHMP} or \cite[Lemma 2.1]{BeBoF} one can show, for $v(r)=\exp(-\alpha r^p)$, that if $V_g$ is bounded on $H^{\infty}_{v}(\C)$, then it coincides with the bi-transpose $(V_g)''$ of $V_g: H_v^0(\C) \rightarrow H_v^0(\C)$. Accordingly, $\sigma(V_g,H_v^0(\C))=\sigma(V_g,H^\infty_v(\C))$.

\section{Spectra of Volterra operators on H\"ormander algebras} \label{spectraalgebras}

A function $p: \mathbb{C} \rightarrow [0, \infty [$ is called a \textit{growth condition} if it  is continuous, subharmonic,  radial, increases with $|z|$  and satisfies:

\begin{itemize}
\item [$(\alpha)$] $\log(1+|z|^2)=o(p(|z|))$ as $|z|\to \infty$,
\item [$(\beta)$] $p(2|z|)=O(p(|z|))$ as $|z|\to \infty.$
\end{itemize}

Given a growth condition $p$, we define the following weighted spaces of entire functions (see e.g.\  \cite{Berenstein_Gay_1995_complex} and \cite{Berenstein_Taylor_1979_a}):
$$A_p:=\{f\in H(\mathbb{C})| \ {\rm there \ is} \ A>0: \sup_{z\in \mathbb{C}}|f(z)| \exp(-Ap(z))<\infty \},$$
endowed with the inductive limit topology, for which it is a (DFN)-algebra (cf.\ \cite{Meise_1985_sequence}). We also define

$$A^0_p:=\{f\in H(\mathbb{C})| \ {\rm for \ all} \ \varepsilon>0: \sup_{z\in \mathbb{C}}|f(z)| \exp(-\varepsilon p(z))<\infty \},$$
endowed with the projective limit topology, for which it is a nuclear Fr\'echet algebra (cf. \cite{Meise_Taylor_1987_sequence}).

Clearly $A^0_p \subset A_p$. Condition $(\alpha)$ implies that the polynomials are dense in $A_p$ and in $A^0_p.$  Condition $(\beta)$ implies that both spaces are stable under differentiation. By the closed graph theorem, the differentiation operator $D$ is continuous on $A_p$ and on $A^0_p.$ It was observed in \cite[Lemma 4.1]{BeBoF2} that the integration operator $J$ is continuous on both spaces.

Weighted algebras of entire functions of this type, usually known as \textit{H\"ormander algebras}, have been considered since the work of Berenstein and Taylor \cite{Berenstein_Taylor_1979_a} by many authors; see e.g.\ \cite{Berenstein_Gay_1995_complex} and the references therein.
 Braun, Meise and Taylor studied in \cite{Braun_1987_weighted}, \cite{Meise_1985_sequence} and \cite{Meise_Taylor_1987_sequence} the structure of
 (complemented) ideals in these algebras. As an example, when $p(z)=|z|^a$, then $A_p$ consists of all entire functions of order $a$ and finite type or order less than $a,$ and $A^0_p$ is the space of all entire functions of order at most $a$ and type $0$. For $a=1$, $A_p$ is the space of all entire functions of exponential type, and $A^0_p$ is the space of  entire functions of infraexponential type.

The following Lemma, that is a consequence of condition $(\beta)$ for the growth condition $p$, is well known.

\begin{lemma}\label{pol}
If $p: \mathbb{C} \rightarrow [0, \infty [$ is a growth condition, then there are $M>0$ and $s>0$ such that $p(r) \leq M r^{s} + M$ for each $r \in [0,\infty[$.
\end{lemma}

\begin{proposition} \label{VolterraHormandercharact} Let $g$ be an entire function.

\begin{itemize}
\item[(i)] $V_g : A_p \rightarrow A_p$ is continuous if and only if $g \in A_p$.

\item[(ii)] $V_g : A^0_p \rightarrow A^0_p$ is continuous if and only if $g \in A^0_p$.

\end{itemize}

\end{proposition}
\begin{proof}
We show (i), the proof of (ii) being similar. The identities $D \circ V_g = M_{g'}$, $V_g = J \circ M_{g'}$ and the continuity of $D$ and $J$ on $A_p$ imply that $V_g : A_p \rightarrow A_p$ is continuous if and only the operator of multiplication $M_{g'}: A_p \rightarrow A_p$ is continuous. Since $A_p$ is an algebra containing the constants, this holds if and only if $g' \in A_p$. We can apply again that $D$ and $J$ are continuous on $A_p$ to conclude that this statement is equivalent to $g \in A_p$.
\end{proof}

The spectrum of the integration operator $J$ on $A_p$ and $A^0_p$ was investigated in \cite[Proposition 4.7]{BeBoF2}. This result, that  corresponds to $V_g$ for $g(z)=z$, is extended in this section.

\begin{lemma}\label{caratheodory}
Let $p: \mathbb{C} \rightarrow [0, \infty [$ be a growth condition and let $h$ be an entire function.
\begin{itemize}

\item[(i)] The function $e^h$ belongs to $A_p$ if and only if $M(h,r)=O(p(r))$ as $r \rightarrow \infty$. If this is the case, then $h$ is a polynomial.

\item[(ii)] The function $e^h$ belongs to $A^0_p$ if and only if $M(h,r)=o(p(r))$ as $r \rightarrow \infty$. If this is the case, then $h$ is a polynomial.

\end{itemize}
\end{lemma}
\begin{proof}
(i) If $M(h,r)=O(p(r))$ as $r \rightarrow \infty$, then $e^h$ clearly belongs to $A_p$. To prove the converse, denote by $A(h,r)$ the maximum of the real part of $h$ in the circle $|z| \leq r$ for $r>0$. As $e^h \in A_p$, there is $C>0$ such that $A(h,r) \leq C p(r) + C$ for each $r > 0$. We apply Carath\'eodory inequality for the circle $|z| \leq 2r$ (c.f.\ \cite[Theorem 6.8]{Levin}) and property $(\beta)$ of $p$ to get
$$
M(h,r) \leq 2(A(h,2r) - {\rm Re} h(0)) + |h(0)| \leq 2Cp(2r) + 2(C + |h(0)|) \leq C' p(r) + C'.
$$
Now Lemma \ref{pol} and a standard argument imply that $h$ is a polynomial.

(ii) By \cite[Lemma 2.3]{Meise_Taylor_1987_sequence}, $e^{h} \in A^0_p$ if and only if there is a growth condition
$q$ such that $e^{h} \in A_{q}$ and $q(r) = o(p(r)), r \rightarrow \infty$. By part (i), $M(h,r) = O(q(r))=o(p(r)), r \rightarrow \infty$.
The reverse implication is trivial.

\end{proof}

\begin{theorem}\label{VolterraAp}
Let $p: \mathbb{C} \rightarrow [0, \infty [$ be a growth condition and let $g \in A_p$ be non-constant.

\begin{itemize}

\item[(i)] If  $M(g,r) = O(p(r)), r \rightarrow \infty,$ is not satisfied (which happens in particular when $p(r)=o(r), r \rightarrow \infty$), then $\sigma(V_g,A_p) = \C$.

\item[(ii)] If  $M(g,r) = O(p(r)), r \rightarrow \infty$, then $\sigma(V_g,A_p) = \{ 0 \}$. In this case $g$
is a polynomial and $r=O(p(r)), r \rightarrow \infty$.

\end{itemize}

Moreover, in both cases we have $\sigma(V_g,A_p) = \{ 0 \} \cup \overline{\{ \lambda \in \C \setminus \{ 0 \} \ | \ e^{\frac{g}{\lambda}} \notin A_p \}} $.
\end{theorem}
\begin{proof}
First observe that $M(g/\lambda,r)=(1/|\lambda|)M(g,r)$ for each $\lambda \in \C \setminus \{ 0 \}$ and $r>0$. Therefore, it follows from
Lemma \ref{caratheodory} (i) that $e^{\frac{g}{\lambda}} \in A_p$ for some (all) $ \lambda \neq 0,$ if and only if $e^g \in A_p$, that is equivalent to
$M(g,r) = O(p(r))$ as $r \rightarrow \infty$.

(i) If $M(g,r) = O(p(r))$ as $r \rightarrow \infty$ is not satisfied, then $e^{\frac{g}{\lambda}} \notin A_p$ for each $\lambda \neq 0$
by our remarks above. We can apply Lemma \ref{subsetspectrumgeneral} to conclude $\sigma(V_g,A_p) = \C$.

Observe that in case $p(r)=o(r), r \rightarrow \infty$, then $M(g,r) = O(p(r))$ as $r \rightarrow \infty$ is not satisfied,
since otherwise we would have $M(g,r) = O(p(r)) =o(r)$ as $r \rightarrow \infty$, that implies that $g$ is constant; a contradiction.

(ii) If $M(g,r) = O(p(r)), r \rightarrow \infty$, then $e^{\frac{g}{\lambda}} \in A_p$ for each $ \lambda \neq 0$. Given $\lambda \in \C\setminus \{ 0 \}$,  setting $G=e^{\frac{g}{\lambda}}, \ 1/G=e^{-\frac{g}{\lambda}}$, the operator $S_{\lambda,g}$ of Lemma \ref{charactresolvent} (iii) satisfies $S_{\lambda,g} = M_G \circ J \circ M_{1/G} \circ D$. These four operators are continuous on the algebra $A_p$. Therefore $\lambda \in \rho(V_g,A_p)$, and $\sigma(V_g,A_p) = \{ 0 \}$.

In this case, since $g$ must be a non constant polynomial, the assumption in (ii) implies $r=O(p(r)), r \rightarrow \infty$.

\end{proof}

\begin{theorem}\label{VolterraA0p}
Let $p: \mathbb{C} \rightarrow [0, \infty [$ be a growth condition and let $g \in A^0_p$ be non-constant.

\begin{itemize}

\item[(i)] If $M(g,r) = o(p(r)), r \rightarrow \infty,$ is not satisfied (which happens in case $p(r)=O(r), r \rightarrow \infty$), then $\sigma(V_g,A^0_p) = \C$.

\item[(ii)] If   $M(g,r) = o(p(r)), r \rightarrow \infty$, then $\sigma(V_g,A^0_p) = \{ 0 \}$. In this case $g$ is a polynomial and $r=o(p(r)), r \rightarrow \infty$.

\end{itemize}

Moreover, in both cases, we have $\sigma(V_g,A^0_p) = \{ 0 \} \cup \overline{\{ \lambda \in \C \setminus \{ 0 \} \ | \ e^{\frac{g}{\lambda}} \notin A^0_p \}} $.
\end{theorem}

\begin{proof}
Since $M(g/\lambda,r)=(1/|\lambda|)M(g,r)$ for each $\lambda \in \C \setminus \{ 0 \}$ and $r>0$, Lemma \ref{caratheodory} (ii) implies
that $e^{\frac{g}{\lambda}} \in A^0_p$ for some (or all) $\lambda \neq 0$ if and only if $e^g \in A^0_p$.

(i) If $M(g,r) = o(p(r))$ as $r \rightarrow \infty$ is not satisfied, then $e^{\frac{g}{\lambda}} \notin A^0_p$ for each $\lambda \neq 0$
by our comments above. By Lemma \ref{subsetspectrumgeneral} we have $\sigma(V_g,A_p) = \C$.

Observe that in case $p(r)=O(r), r \rightarrow \infty$, then $M(g,r) = o(p(r))$ as $r \rightarrow \infty$ does not hold.
Otherwise, $M(g,r) = o(p(r)) =O(r)$ as $r \rightarrow \infty$, which implies that $g$ is constant; a contradiction.

(ii) If $M(g,r) = o(p(r)), r \rightarrow \infty$, then $e^{\frac{g}{\lambda}} \in A^0_p$ for each $ \lambda \neq 0$ by our comments above. Given $\lambda \in \C\setminus \{ 0 \}$,  setting $G=e^{\frac{g}{\lambda}}$, the operator $S_{\lambda,g}$ of Lemma \ref{charactresolvent} (iii) satisfies $S_{\lambda,g} = M_G \circ J \circ M_{1/G} \circ D$ and is continuous on the algebra $A^0_p$. Therefore $\lambda \in \rho(V_g,A^0_p)$.

Since $g$ must be a non constant polynomial, the assumption in (ii) yields $r=o(p(r))$ as $r \rightarrow \infty$.

\end{proof}

\vspace{.3cm}

\textbf{Acknowledgement.} This research was partially supported by MINECO Project MTM2013-43540-P and by GV Project Prometeo II/2013/013.

%%%%%%%%%%%%%%%%%%%%%%%%%%%%%%%%%%%%%%%%%%%%%%%%%%%%%%%%%%%%%%%%%%%%%%%%%
%%%%%%%%%%%%%%%%%%%%%%%%%%%%%%%%%%%%%%%%%%%%%%%%%%%%%%%%%%%%%%%%%%%%%%%%%
%%% Bibliography
%%%%%%%%%%%%%%%%%%%%%%%%%%%%%%%%%%%%%%%%%%%%%%%%%%%%%%%%%%%%%%%%%%%%%%%%%
%%%%%%%%%%%%%%%%%%%%%%%%%%%%%%%%%%%%%%%%%%%%%%%%%%%%%%%%%%%%%%%%%%%%%%%%%

%%%%%%%%%%%%%%%%%%%%%%%%%%%%%%%%%%%%%%%%%%%%%%%%%%%%%%%%%%%%%%%%%%%%%%%%%
%%%%%%%%%%%%%%%%%%%%%%%%%%%%%%%%%%%%%%%%%%%%%%%%%%%%%%%%%%%%%%%%%%%%%%%%%
%%% Other stuff. Otras cosas
%%%%%%%%%%%%%%%%%%%%%%%%%%%%%%%%%%%%%%%%%%%%%%%%%%%%%%%%%%%%%%%%%%%%%%%%%
%%%%%%%%%%%%%%%%%%%%%%%%%%%%%%%%%%%%%%%%%%%%%%%%%%%%%%%%%%%%%%%%%%%%%%%%%

\noindent \textbf{Author's address:}%
\vspace{\baselineskip}%

Jos\'e Bonet: Instituto Universitario de Matem\'{a}tica Pura y Aplicada IUMPA,
Universitat Polit\`{e}cnica de Val\`{e}ncia,  E-46071 Valencia, Spain

email: jbonet@mat.upv.es \\

\end{document}